\documentclass[12pt]{article}
\usepackage{amssymb,amsmath,latexsym}
\usepackage{amsfonts,graphics,epsfig,cite}
\usepackage{amsmath,graphicx,amssymb,amsthm}
\usepackage[margin=1.1in]{geometry}
\usepackage{longtable}
\usepackage{tikz}
\usepackage{makeidx}
\usetikzlibrary{decorations.pathmorphing}
\usetikzlibrary{decorations.pathreplacing}
\usetikzlibrary{patterns}
\usetikzlibrary{arrows}

\usepackage[skins,theorems]{tcolorbox}
\tcbset{highlight math style={enhanced,
  colframe=red,colback=white,arc=0pt,boxrule=1pt}}

\theoremstyle{plain}
\newtheorem{theorem}{Theorem}[section]

\newtheorem{lemma}[theorem]{Lemma}

\newtheorem{problem}{Problem}[section]

\newcommand{\Ga}{\ensuremath{\Gamma}}

\makeindex

\begin{document}

\title{Connectivity concerning the last two subconstituents 
\\
of a $Q$-polynomial distance-regular graph}
\author{S.M. Cioab\u{a}\footnote{Department of Mathematical Sciences, University of Delaware, Newark, DE 19716-2553, USA {\tt cioaba@udel.edu}}\,\, J.H. Koolen\footnote{School of Mathematical Sciences, University of Science and Technology of China,
Wen-Tsun Wu Key Laboratory of the Chinese Academy of Sciences, Anhui, 230026, China, {\tt koolen@ustc.edu.cn}}\,\, P. Terwilliger\footnote{Department of Mathematics, University of Wisconsin, 480 Lincoln Drive, Madison, WI 53706-1388, USA, {\tt terwilli@math.wisc.edu}}}
\date{\today}
\maketitle

\begin{abstract}
Let $\Ga$ be a $Q$-polynomial distance-regular graph of diameter $d\geq 3$. Fix a vertex $\gamma$ of $\Ga$ and consider the
subgraph induced on the union of the last two subconstituents of $\Ga$ with respect to $\gamma$. We prove that this subgraph is connected.
\end{abstract}

\section{Introduction}

\hspace{1em} All the graphs considered here will be finite and undirected, with no loops nor multiple edges. We briefly review the key definitions and basic results involving distance-regular graphs. For other notations and definitions, see \cite{BCN,BH,DKT}. Let $\Ga$ be a connected graph with vertex set $X$. For $x,y\in X$, the distance between $x$ and $y$ is denoted by $\partial(x,y)$, and any path between $x$ and $y$ of length $\partial(x,y)$ is called {\em geodesic}. The {\em diameter} $\max_{x,y\in X}\partial(x,y)$ of $\Ga$ is denoted by $d$. For an integer $k\geq 0$, $\Ga$ is said to be {\it regular with valency $k$} whenever each vertex of $\Ga$ is adjacent to exactly $k$ vertices of $\Ga$. The graph $\Ga$ is called {\em distance-regular} 
whenever for all integers $0\leq h,i,j\leq d$ there exists a nonnegative 
integer $p^{h}_{ij}$ such that for all $x,y\in X$ with $\partial(x,y)=h$, 
\begin{equation*}
p^{h}_{ij}=|\{z\in X: \partial(z,x)=i,\partial(z,y)=j\}|.
\end{equation*}
For the rest of this paper we assume that $\Ga$ is distance-regular of diameter $d\geq 2$. Note that $\Ga$ is regular with valency $k=p^0_{11}$; to avoid trivialities we always assume $k\geq 3$. Let $A_0,A_1,\dots,A_d$ denote the distance
matrices of $\Ga$ (see \cite[p.127]{BCN}). Then $A_0,A_1,\dots,A_d$ form a
basis for a commutative semisimple $\mathbb{R}$-algebra $M$ known as the {\it Bose-Mesner algebra} of $\Ga$. The algebra $M$ has a second basis $E_0,E_1,\dots,E_d$ such that 
\begin{align*}
E_iE_j&=\delta_{ij}E_i\quad (0\leq i,j\leq d),\\
I&=E_0+\dots +E_d,\\
E_0&=|X|^{-1}J,
\end{align*}
where $I$ is the identity matrix and $J$ is the all ones matrix (see \cite[Thm 2.6.1]{BCN}). We refer to $E_0,E_1,\dots,E_d$ as the {\em primitive idempotents} of $\Ga$. The primitive idempotent $E_0$ is called {\em trivial}. The ordering
$E_0,E_1,\dots,E_d$ is said to be {\em $Q$-polynomial} whenever
for $0 \leq i \leq d$ there exists a polynomial $q_i$ of 
degree $i$ such that $E_i=q_i(E_1)$ 
(where the matrix multiplication is done entry-wise).
For a primitive idempotent $E$ of $\Gamma$,
we say that $\Gamma$ is {\em $Q$-polynomial with respect to $E$} 
whenever there exists a $Q$-polynomial ordering $E_0,E_1,\dots,E_d$
of the primitive idempotents 
such that $E=E_1$. The  graph $\Gamma$ is called {\em $Q$-polynomial} 
whenever it is $Q$-polynomial with respect to at least one primitive idempotent. 

We now recall the antipodal property. Define a binary relation $\sim$ on $X$ such that for all $x, y \in X$, $x \sim y$
whenever $x=y$ or $\partial(x,y)=d$. The graph $\Gamma$ is called {\em antipodal} whenever $\sim$ is an equivalence relation. 
The graph $\Gamma$ is said to be {\em primitive} whenever $\Gamma$ is not bipartite nor antipodal
(see \cite[Thm 4.2.1]{BCN}).
A long-standing conjecture of Bannai and Ito \cite[p.~312]{BI} states that 
if $\Gamma$ is primitive and $d$ is sufficiently large,
then $\Gamma$ is $Q$-polynomial. 
For more information about the $Q$-polynomial property,
see \cite{BI,BCN} or \cite[Chapter 5]{DKT}. 

For $0\leq i\leq d$ and $\gamma\in X$, let $\Gamma_i(\gamma)$
denote the set of vertices in $\Gamma$ at distance $i$
from $\gamma$. The subgraph induced by $\Gamma_i(\gamma)$ is 
called the {\em $i$-th subconstituent of $\Gamma$ with respect to $\gamma$}. 
Combinatorial and algebraic properties of these subconstituents have been 
studied by several authors (see \cite{CGS,CK,GGHR} for example).
The graph $\Gamma$ is called {\it strongly-regular} whenever $d=2$.
If $\Gamma$ is strongly-regular and primitive,
then with respect to any vertex 
the second subconstituent 
of $\Gamma$
is connected. 
See \cite[p.~126]{BH} for an algebraic proof, and 
\cite{GGHR} for a combinatorial proof. 
Answering a question of Brouwer \cite{Br},
Cioab\u{a} and Koolen \cite{CK} generalized this
result in the following way.
Consider the dual eigenvalue sequence
$\theta^*_0, \theta^*_1,\ldots, \theta^*_d$ for the second largest eigenvalue of $\Gamma$
(see equation \eqref{dualeig} for a definition).
By \cite[Ch. 4]{BCN} there exists a unique integer
$s$ $(1 \leq s \leq d)$
such that $\theta^{*}_{s-1}>0$ and $\theta^{*}_s\leq 0$.
Then for any vertex $\gamma$ of $\Gamma$
the subgraph induced on $\cup_{i=s}^d\Gamma_i(\gamma)$ is connected
\cite{CK}.
In \cite{CK} the authors also prove that $s\geq d/2$ and
pose the following problem. 
\begin{problem}[Cioab\u{a}-Koolen \cite{CK}]
Assume that $\Gamma$ is primitive and $d\geq 3$. Is it true that for 
any vertex $\gamma$, the subgraph induced on $\Gamma_{d-1}(\gamma)\cup \Gamma_{d}(\gamma)$ is connected ?
\end{problem}
In \cite{CK}, this was shown to be true if $d\in \{3,4\}$. 
In this note, we show that it is true for all $d\geq 3$, provided 
that  $\Gamma$ is $Q$-polynomial. We now state our main result. 
\begin{theorem}\label{main}
Let $\Gamma$ be a $Q$-polynomial distance-regular graph of diameter 
$d\geq 3$. Then for any vertex $\gamma$ of $\Gamma$ 
the subgraph induced on $\Gamma_{d-1}(\gamma)\cup \Gamma_d(\gamma)$ 
is connected.
\end{theorem}
The main tool for our proof is Terwilliger's balanced set condition 
(see \cite{T,T2} or Theorem \ref{Ter} in the next section). 
This condition has been used   
by Lewis \cite{Lewis} to prove that the girth is at most 6 for any 
$Q$-polynomial distance-regular graph of 
valency at least $3$.

\section{Proof of the main result}

\hspace{1em} For a primitive idempotent $E$ of $\Ga$, there exist real numbers $\theta_0^{*},\theta_1^{*},\dots, \theta_d^{*}$ (called the dual eigenvalues of $\Ga$ with respect to $E$) such that 
\begin{equation}\label{dualeig}
E=|X|^{-1}\sum_{h=0}^{d}\theta_h^{*}A_h.
\end{equation}
We equip the vector space $\mathbb{R}^{X}$ with an inner
product such that 
$\langle u,v\rangle=u^{t}v$ for all $u,v\in \mathbb{R}^{X}$.
For $x\in X$, let $\hat{x}$ denote the vector in $\mathbb{R}^{X}$ with $x$-coordinate $1$ and all other coordinates $0$. Equation \eqref{dualeig} implies that 
\begin{equation}
\langle E\hat{x}, E\hat{y}\rangle=|X|^{-1}\theta_i^{*},
\end{equation}
where $i=\partial(x,y)$. The main tool for our proof is the following theorem.
\begin{theorem}[Terwilliger \cite{T,T2}]\label{Ter}
Let $\Gamma$ be a distance-regular graph with diameter $d\geq 3$,
and let $E$ denote a nontrivial primitive idempotent of $\Gamma$ with dual 
eigenvalues $\theta^*_0, \theta^*_1,\ldots, \theta^*_d$.
Then $\Ga$ is $Q$-polynomial with respect to $E$ if and only if $\theta_0^*\notin\{\theta_1^*,\dots,\theta_d^*\}$ and 
\begin{equation}\label{balset}
\sum_{z\in \Ga_i(x)\cap \Ga_j(y)}E\hat{z}-\sum_{w\in \Ga_j(x)\cap \Ga_i(y)}E\hat{w}=p^h_{ij}\frac{\theta_i^*-\theta_j^*}{\theta_0^*-\theta_h^*}(E\hat{x}-E\hat{y})
\end{equation}
for all integers $h,i,j$ with $1\leq h\leq d$ and $0\leq i,j\leq d$
and all vertices $x,y$ with $\partial(x,y)=h$. Furthermore, if the 
conditions above hold, then $\theta_0^*,\theta_1^*,\dots,\theta_d^*$ are 
mutually distinct. \end{theorem}
The equation \eqref{balset} is usually called the balanced set condition. We are now ready to give the proof of Theorem \ref{main}.
\begin{proof}[Proof of Theorem \ref{main}] 
Let $E$ be a primitive idempotent of $\Gamma$ with respect to which
$\Gamma$ is $Q$-polynomial.
We will use a proof by contradiction,
and assume that there exists $\gamma \in X$ such that
the subgraph induced on
$\Gamma_{d-1}(\gamma)\cup \Ga_{d}(\gamma)$ is disconnected.
Let $C$ be the vertex set of a connected
component of the subgraph induced on $\Ga_{d-1}(\gamma)\cup \Ga_d(\gamma)$. 
Let the set $\Delta$ consist of the vertices in $X$
that lie on a geodesic from $\gamma$ to $C$. 
The set $\Delta$ is properly contained in $X$ since
$C\neq \Gamma_{d-1}(\gamma)\cup \Gamma_d(\gamma)$.
We partition $\Delta = \cup_{j=0}^d \Delta_j$ where 
 $\Delta_j = \Delta\cap \Gamma_{j}(\gamma)$
for $0 \leq j \leq d$. Note that for 
 $0 \leq j\leq d-1$, each vertex in $\Delta_j$ has at least one
neighbor in $\Delta_{j+1}$.

A vertex in $\Delta$ will be called a {\it border} whenever it
is adjacent to a vertex in 
$X \setminus \Delta$.
Since $\Delta \not=X$ and $\Gamma$ is connected,
$\Delta$ contains at least one border vertex.
Let $t$ denote the maximal integer $j$ $(0 \leq j \leq d)$
such that $\Delta_j$ contains a border vertex.
By the construction $1 \leq t \leq d-2$.

Pick a border vertex $z \in \Delta_t$.
There exists $x\in \Delta_{t+2}$ such that
$\partial(x,z)=2$. Let $y\in X\setminus \Delta$ be a neighbor of $z$.
Define $\xi = \partial (\gamma, y)$.
By the triangle inequality $\xi \in \lbrace t-1,t,t+1\rbrace$.
Note that $\xi \not=t-1$;
otherwise $y$ is on a geodesic from $\gamma$ to $C$
passing through $z$, forcing $y \in \Delta$ for a contradiction.
Therefore $\xi = t$ or $\xi=t+1$. 


We next show that $\partial(x,y)=3$.
Because $\partial(x,z)=2$ and $\partial(z,y)=1$, the triangle inequality
implies that $\partial(x,y)\leq 3$.
By the maximality of $t$ and since $x \in \Delta_{t+2}$,
we see that $x$ is not a border and not adjacent to a border. Therefore $\Delta$ contains all
the vertices of $\Ga$ that are at distance at most $2$ from $x$.
The vertex $y$ is not in $\Delta$, so $\partial(x,y)\geq 3$. We have shown that $\partial(x,y)=3$.

%

Note that  $\Ga_1(x)\cap \Ga_2(y)\subset \Ga_{t+1}(\gamma)$ and 
$\Ga_2(x)\cap \Ga_1(y)\subset \Ga_{t}(\gamma)$. We apply the balanced set condition \eqref{balset} to $x$ and $y$
using $h=3,i=1,j=2$
and then take the inner product of each side
with $E\hat{\gamma}$; this gives
\begin{align}\label{eq1case1}
p^{3}_{12}(\theta^*_{t+1}-\theta^{*}_{t})=p^3_{12}\frac{\theta^{*}_1-\theta^{*}_2}{\theta^{*}_0-\theta^{*}_3}(\theta^*_{t+2}-\theta^{*}_{\xi}).
\end{align}
There exists $y'\in \Ga_{t-1}(\gamma)\cap \Ga_1(z)$. 
We have $\partial(x,y')=3$ and 
$\Ga_1(x)\cap \Ga_2(y')\subset \Ga_{t+1}(\gamma)$
and $\Ga_2(x)\cap \Ga_1(y')\subset \Ga_{t}(\gamma)$.
We apply the balanced set condition
\eqref{balset} to $x$ and $y'$ using $h=3,i=1,j=2$ and then
take the inner product of each side with $E\hat{\gamma}$; this gives
\begin{equation}\label{eq2case2}
p^{3}_{12}(\theta^{*}_{t+1}-\theta^{*}_{t})=p^3_{12}\frac{\theta^{*}_1-\theta^{*}_2}{\theta^{*}_0-\theta^{*}_3}(\theta^{*}_{t+2}-\theta^{*}_{t-1}).
\end{equation}
Comparing 
\eqref{eq1case1} and \eqref{eq2case2} we obtain
$\theta^*_{\xi}=\theta^*_{t-1}$. We have $\xi=t-1$
since $\theta^*_0, \theta^*_1,\ldots, \theta^*_d$
are mutually distinct. We mentioned earlier that $\xi\not=t-1$,
for a contradiction.
We conclude that the subgraph induced on
$\Gamma_{d-1}(\gamma)\cup \Gamma_d(\gamma)$ 
is connected.
\end{proof}

To see how Theorem \ref{main} is best possible, assume that $\Ga$ is the 
Odd graph $O_{d+1}$ with $d\geq 3$. Recall that the 
vertices of $\Ga$ are the $d$-subsets of a set $\Omega$ of size $2d+1$. 
Two vertices $\alpha$ and $\beta$ are adjacent whenever
$\alpha\cap \beta=\emptyset$. The diameter of $\Ga$ is $d$ and 
its intersection numbers are known (see \cite{Biggs} or
\cite[Prop 9.1.7]{BCN}). For $0\leq h\leq d$, we have $p^{h}_{1h}=0$ 
if $h<d$ and $p^{h}_{1h}=\lceil \frac{d+1}{2}\rceil$ if $h=d$. 
So with respect to any vertex of $\Ga$, the $h$-th subconstituent has no 
edges if $h<d$ and is regular with valency $\lceil \frac{d+1}{2}\rceil$
if $h=d$. 
\begin{lemma}
Assume that $\Ga$ is the Odd graph $O_{d+1}$ with $d\geq 3$. For
any $\gamma\in X$, the number of connected components in the
$d$-th subconstituent of $\Ga$ with respect to $\gamma$ is 
equal to $\binom{2m}{m}/2$, 
where $m=d/2$ if $d$ is even and $m=(d+1)/2$ if $d$ is odd. Moreover, this $d$-th subconstituent is not connected.
\end{lemma}
\begin{proof} From the intersection numbers of $\Gamma$
we obtain
 $\vert \Gamma_d(\gamma)\vert=\binom{d}{m}\binom{d+1}{m}$. 
Using the results of Biggs
\cite{Biggs},
each connected component of $\Gamma_d(\gamma)$ is
isomorphic to the bipartite double (see \cite[Section 1.11]{BCN})
of $O_{r+1}$, where $r=d/2$ if $d$ is even
and $r=(d-1)/2$ if $d$ is odd. This bipartite double has 
 $2 \binom{2r+1}{r}$ vertices. The result follows after some routine algebra. Note that the lemma also follows by observing that $\Gamma_d(\gamma)$ consists of the vertices at distance $m$ from $\gamma$ in the Johnson graph $J(2d+1,d)$.
\end{proof}
Note also that for $O_{d+1}$ the subgraph induced on 
$\Ga_{1}(\gamma)\cup \Ga_2(\gamma)$ is disconnected. 
Next assume that $\Gamma$ is the folded $(2d+1)$-cube.
It has diameter $d$ and for $1\leq h\leq d-1$,
the $h$-subconstituent of $\Ga$ with respect to any vertex has no edges
(see \cite[p.~264]{BCN}), and consequently not connected.
Gardiner, Godsil, Hensel and Royle \cite{GGHR} proved that the diameter of 
the second subconstituent of a primitive strongly-regular graph is at 
most three. It would be interesting to extend this result
to distance-regular graphs with diameter $d\geq 3$. For example, if
$\Ga$ is a distance-regular with  $d=3$,
then what is the diameter of $\Ga_3(\gamma)$ 
when $\Ga_3(\gamma)$ is connected ?
Another related problem from \cite{CK} is to classify the distance-regular graphs $\Ga$ of diameter $3$
such that $\Ga_3(\gamma)$ is disconnected for some vertex $\gamma$.
See \cite{KP} for related results.

The vertex-connectivity of a primitive distance-regular graph 
is equal to its valency,
as proved by Brouwer and Mesner \cite{BM} for diameter $d=2$,
and by Brouwer and Koolen \cite{BK} for $d\geq 3$. 
Brouwer and Haemers \cite[p.~127]{BH} observed that for
certain strongly-regular graphs constructed by Haemers \cite[p.~76]{H} 
the vertex-connectivity of their 
second subconstituent is strictly less than the valency. 
It would be interesting to determine lower bounds for the
vertex-connectivity and edge-connectivity of 
the subconstituents for a distance-regular graph with $d\geq 3$.
See \cite{B,CKK,CKL,CKL2,G,KM} for related connectivity results
concerning distance-regular graphs and association schemes.

\section*{Acknowledgments} The authors thank the referees for useful comments and suggestions. The research of the first author was
supported by the grants NSF DMS-1600768, CIF-1815922, and a JSPS
Invitational Fellowship for Research in Japan (Short-term S19016).
The research of the second author is partially supported by the
National Natural Science Foundation of China (Grant No. 11471009 and
Grant No. 11671376) and by Anhui Initiative in Quantum Information 
Technologies (Grant No. AHY150200). 
Part of this work was done while the authors were visiting Anhui University, Hefei, China. We 
thank Yi-Zheng Fan, Tatsuro Ito and their students for wonderful hospitality.

\end{document}